\documentclass[11pt]{amsart}
\usepackage{amsfonts,amssymb,amsmath,amsthm}
\usepackage{url}
\usepackage{enumerate}

\theoremstyle{plain}
\newtheorem{theorem}{Theorem}
\newtheorem*{theorem*}{Theorem}

\newtheorem{corollary}{Corollary}

\newtheorem{example}{Example}

\numberwithin{equation}{section}

\begin{document}

\title{Generalized Stirling transform}

\author{Mourad Rahmani}
\address{USTHB, Faculty of Mathematics,
P. O. Box 32, El Alia,16111, Algiers, Algeria}
\email{mrahmani@usthb.dz}

\begin{abstract}
In this paper, algorithms are developed for computing the Stirling
transform and the inverse Stirling transform; specifically, we
investigate a class of sequences satisfying a two-term recurrence.
We derive a general identity which generalizes the usual Stirling
transform and investigate the corresponding generating functions
also. In addition, some interesting consequences of these results related
to classical sequences like Fibonacci, Bernoulli and the numbers of
derangements have been derived.
\end{abstract}
\maketitle

\section{Introduction}
The Stirling numbers arise frequently in mathematics, especially in
enumerative problems. This is the reason of their important role in
combinatorial analysis, number theory, probability, graph theory,
calculus of finite differences and interpolation. The notations for
these numbers have never been standardized, this paper follows the
notation of Riordan for the signed Stirling numbers of the first
kind $s\left(  n,k\right)  $ and Knuth's
notation for the Stirling numbers of the second kind $%
\genfrac{\{}{\}}{0pt}{}{n}{k}%
$. \vskip 3mm

The Stirling transform of a sequence $\left(  a_{n}\right)  $ is the the
sequence $\left(  b_{n}\right)  $ given by
\begin{equation}
b_{n}=%
{\displaystyle\sum\limits_{k=0}^{n}}
\genfrac{\{}{\}}{0pt}{}{n}{k}%
a_{k},\label{1}%
\end{equation}
and the inverse transform is
\begin{equation}
a_{n}=%
{\displaystyle\sum\limits_{k=0}^{n}}
s\left(  n,k\right)  b_{k}\label{2}.
\end{equation}

The identity (\ref{1}) has a combinatorial interpretation given in \cite{sloane1}. If $a_{n}$ is the number of objects in some class with points
labeled $1,2,\ldots,n$ (with all labels distinct) then $b_{n}$ is the number of
objects with points labeled $1,2,\ldots,n$ (with repetitions allowed).

In this paper, algorithms are developed for computing the Stirling
transform and the inverse Stirling transform; specifically, we
investigate a class of sequences satisfying a two-term recurrence.
We derive a general identity which generalizes the usual Stirling
transform and investigate the corresponding generating functions
also. \vskip 2mm

Given a sequence $a_{m}:=a_{0,m}$ $(m\geq0)$. We construct an
infinite matrix $\mathcal{S}:=\left(  a_{n,m}\right)  $ as follows:
\vskip 2mm

The first row $a_{0,m}$ of the matrix is the initial sequence; the
first column $b_{n}:=a_{n,0}$ $(n\geq0)$ is called the final
sequence and, each
entry $a_{n,m}$ is given recursively by%
\begin{equation}
a_{n+1,m}=a_{n,m+1}+ma_{n,m}.\label{rec1}%
\end{equation}

Conversely, if we start with the final sequence, the matrix $\mathcal{S}$ can be
recovered by the recursive relations%
\begin{equation}
a_{n,m+1}=a_{n+1,m}-ma_{n,m}.\label{rec2}%
\end{equation}

\section{Definitions and notation}
In this section, we introduce some definitions and notations which
are useful in the rest of the paper. $\mathbb{N}$ being the set of
positive integers and $\mathbb{N}_{0}=\mathbb{N\cup}\left\{
0\right\}  .$

The falling and rising factorials are defined, respectively by%
\[
\left(  x\right)  _{n}=x\left(  x-1\right)  \cdots\left(  x-n+1\right)
,\left(  x\right)  _{0}=1
\]
and%
\[
\left\langle x\right\rangle _{n}=x\left(  x+1\right)  \cdots\left(
x+n-1\right)  ,\left\langle x\right\rangle _{0}=1.
\]

The (signed) Stirling numbers $s(n,k)$ of the first kind, which are usually
defined by%
\begin{equation}
\left(  x\right)  _{n}=%
{\displaystyle\sum\limits_{k=0}^{n}}
s\left(  n,k\right)  x^{k},\label{defstir1}%
\end{equation}
or by the following generating function%
\begin{equation}
\frac{1}{k!}\left(  \ln\left(  1+x\right)  \right)  ^{k}=%
{\displaystyle\sum\limits_{n\geq k}}
s\left(  n,k\right)  \frac{x^{n}}{n!}.\label{gensti1}%
\end{equation}
It follows from (\ref{defstir1}) or (\ref{gensti1}) that%
\begin{equation}
s\left(  n+1,k\right)  =s\left(  n,k-1\right)  -n\text{ }s\left(  n,k\right)\label{recs}
\end{equation}
and that%
\[
s(n,0)=\delta_{n,0}\text{ \ }\left(  n\in\mathbb{N}\right)  ,\text{ }s\left(
n,k\right)  =0\text{ \ }(k>n\text{ or }k<0),
\]
where $\delta_{n,m}$ denotes the Kronecker symbol.

The Stirling numbers $%
\genfrac{\{}{\}}{0pt}{}{n}{k}%
$ of the second kind count the number of possible partitions of a set of $n$
objects into $k$ disjoint blocks. These numbers can be defined explicitly by%
\[
x^{n}=%
{\displaystyle\sum\limits_{k=0}^{n}}
\genfrac{\{}{\}}{0pt}{}{n}{k}%
\left(  x\right)  _{k}.
\]

For any positive $r\in\mathbb{N}$ the quantity $%
\genfrac{\{}{\}}{0pt}{}{n}{k}%
_{r}$ denotes the number of partitions of a set of $n$ objects into exactly
$k$ nonempty, disjoint subsets, such that the first $r$ elements are in
distinct subsets. These numbers obey the recurrence relation
\begin{equation}
\begin{tabular}
[c]{lll}%
$
\genfrac{\{}{\}}{0pt}{0}{n}{k}%
_{r}=0,$ &  & $n<r,$\\
$
\genfrac{\{}{\}}{0pt}{0}{n}{k}%
_{r}=\delta_{k,r},$ &  & $\text{ }n=r,$\\
$%
\genfrac{\{}{\}}{0pt}{0}{n}{k}%
_{r}=k%
\genfrac{\{}{\}}{0pt}{0}{n-1}{k}%
_{r}+%
\genfrac{\{}{\}}{0pt}{0}{n-1}{k-1}%
_{r},$ &  & $n>r,$\label{tab}
\end{tabular}
\end{equation}
The exponential generating function is given by
\begin{equation}
{\displaystyle\sum\limits_{n\geq k}}
\genfrac{\{}{\}}{0pt}{0}{n+r}{k+r}%
_{r}\frac{x^{n}}{n!}=\frac{1}{k!}e^{rx}\left(  e^{x}-1\right)  ^{k}.\label{rstigen}
\end{equation}
The properties
\[
\genfrac{\{}{\}}{0pt}{0}{n}{r}%
_{r}=r^{n-r}
\]
and
\begin{equation}
\genfrac{\{}{\}}{0pt}{0}{n+r}{k+r}%
_{r}=%
\genfrac{\{}{\}}{0pt}{0}{n+r}{k+r}%
_{r-1}-(r-1)%
\genfrac{\{}{\}}{0pt}{0}{n+r-1}{k+r}%
_{r-1}\label{tig}
\end{equation}
are given in \cite{Broder}, which one can consult for more details on $r$-Stirling numbers.

\section{Combinatorial identities}

\begin{theorem}\label{th1}
Given an initial sequence $\left(  a_{0,m}\right)  _{m\geq0},$
define the matrix $\mathcal{S}$ by (\ref{rec1}). Then, the entries
of the infinite matrix
$\mathcal{S}$ are given by%
\begin{equation}
a_{n,m}=%
{\displaystyle\sum\limits_{k=0}^{n}}
\genfrac{\{}{\}}{0pt}{}{n+m}{k+m}%
_{m}a_{0,m+k}.
\end{equation}

\end{theorem}

\begin{proof}
We proof by induction on $n$, the result clearly holds for $n=0$. By induction
hypothesis
\begin{align*}
a_{n,m+1}+ma_{n,m} &  ={\displaystyle\sum\limits_{k=0}^{n}}%
%TCIMACRO{\QATOPD{\{}{\}}{n+m+1}{k+m+1}}%
%BeginExpansion
\genfrac{\{}{\}}{0pt}{}{n+m+1}{k+m+1}%
%EndExpansion
_{m+1}a_{0,m+k+1}+m%
%TCIMACRO{\QATOPD{\{}{\}}{n+m}{m}}%
%BeginExpansion
\genfrac{\{}{\}}{0pt}{}{n+m}{m}%
%EndExpansion
_{m}a_{0,m}\\
&  +m{\displaystyle\sum\limits_{k=1}^{n-1}}%
%TCIMACRO{\QATOPD{\{}{\}}{n+m}{k+m}}%
%BeginExpansion
\genfrac{\{}{\}}{0pt}{}{n+m}{k+m}%
%EndExpansion
_{m}a_{0,m+k}\\
&  ={\displaystyle\sum\limits_{k=0}^{n}}%
%TCIMACRO{\QATOPD{\{}{\}}{n+m+1}{k+m+1}}%
%BeginExpansion
\genfrac{\{}{\}}{0pt}{}{n+m+1}{k+m+1}%
%EndExpansion
_{m+1}a_{0,m+k+1}+m%
%TCIMACRO{\QATOPD{\{}{\}}{n+m}{m}}%
%BeginExpansion
\genfrac{\{}{\}}{0pt}{}{n+m}{m}%
%EndExpansion
_{m}a_{0,m}\\
&  +m{\displaystyle\sum\limits_{k=0}^{n-2}}%
%TCIMACRO{\QATOPD{\{}{\}}{n+m}{k+m+1}}%
%BeginExpansion
\genfrac{\{}{\}}{0pt}{}{n+m}{k+m+1}%
%EndExpansion
_{m}a_{0,m+k+1}\\
&  =%
%TCIMACRO{\QATOPD{\{}{\}}{n+m+1}{n+m+1}}%
%BeginExpansion
\genfrac{\{}{\}}{0pt}{}{n+m+1}{n+m+1}%
%EndExpansion
_{m+1}a_{0,m+n+1}+%
%TCIMACRO{\QATOPD{\{}{\}}{n+m+1}{n+m}}%
%BeginExpansion
\genfrac{\{}{\}}{0pt}{}{n+m+1}{n+m}%
%EndExpansion
_{m+1}a_{0,m+n}\\
&  +{\displaystyle\sum\limits_{k=0}^{n-2}}%
%TCIMACRO{\QATOPD{\{}{\}}{n+m+1}{k+m+1}}%
%BeginExpansion
\genfrac{\{}{\}}{0pt}{}{n+m+1}{k+m+1}%
%EndExpansion
_{m+1}a_{0,m+k+1}\\
&  +m%
%TCIMACRO{\QATOPD{\{}{\}}{n+m}{m}}%
%BeginExpansion
\genfrac{\{}{\}}{0pt}{}{n+m}{m}%
%EndExpansion
_{m}a_{0,m}+m{\displaystyle\sum\limits_{k=0}^{n-2}}%
%TCIMACRO{\QATOPD{\{}{\}}{n+m}{k+m+1}}%
%BeginExpansion
\genfrac{\{}{\}}{0pt}{}{n+m}{k+m+1}%
%EndExpansion
_{m}a_{0,m+k+1}.
\end{align*}
From (\ref{tig}) and after some rearrangements, we get  %
\begin{align*}
a_{n,m+1}+ma_{n,m}  & =%
{\displaystyle\sum\limits_{k=0}^{n+1}}
\genfrac{\{}{\}}{0pt}{}{n+m+1}{k+m}%
_{m}a_{0,m+k}.\\
& =a_{n+1,m}.
\end{align*}

\end{proof}

\begin{theorem}\label{th2}
Given a final sequence $\left(  a_{n,0}\right)  _{n\geq0},$ define the
matrix $\mathcal{S}$ by (\ref{rec2}). Then, the entries of the
infinite matrix
$\mathcal{S}$ are given by%
\begin{equation}
a_{n,m}=%
{\displaystyle\sum\limits_{k=0}^{m}}
s\left(  m,k\right)  a_{n+k,0}\label{sah}.
\end{equation}

\end{theorem}

\begin{proof}
We proof by induction on $m$, the result clearly holds for $n=0$. By induction
hypothesis and (\ref{recs}), we have
\begin{align*}
a_{n+1,m}-ma_{n,m} &  ={\displaystyle\sum\limits_{k=1}^{m+1}}s\left(
m,k-1\right)  a_{n+k,0}-m{\displaystyle\sum\limits_{k=0}^{m}}s\left(
m,k\right)  a_{n+k,0}\\
&  =s\left(  m,m\right)  a_{n+m+1,0}+{\displaystyle\sum\limits_{k=1}^{m}%
}s\left(  m,k-1\right)  a_{n+k,0}\\
&  -ms\left(  m,0\right)  a_{n,0}-m{\displaystyle\sum\limits_{k=1}^{m}%
}s\left(  m,k\right)  a_{n+k,0}\\
&  =s\left(  m,m\right)  a_{n+m+1,0}+{\displaystyle\sum\limits_{k=1}^{m}%
}\left(  s\left(  m,k-1\right)  -ms\left(  m,k\right)  \right)  a_{n+k,0}\\
&  -ms\left(  m,0\right)  a_{n,0}\\
&  =a_{n,m+1}.
\end{align*}

\end{proof}

\begin{corollary}%
\begin{equation}%
{\displaystyle\sum\limits_{k=0}^{m}}
s\left(  m,k\right)  b_{n+k}=%
{\displaystyle\sum\limits_{k=0}^{n}}
\genfrac{\{}{\}}{0pt}{}{n+m}{k+m}%
_{m}a_{m+k}.\label{ega}%
\end{equation}

\end{corollary}

The last identity can be viewed as the generalized Stirling
transform which reduced, for $m=0,$ to the Stirling transform (\ref{1}) of the
sequence $a_{n}$, and for $n=0$ reduces to the inverse Stirling transform (\ref{2}) of the sequence
$b_{m}$. We may now formulate the following algorithms

\textbf{Algorithm 1. }Stirling transform

\textbf{Input:} $a_{n}$

\textbf{Output:} $b_{n}$

Set $X_{m}=a_{n-m,m}$

\textbf{for} $n=0,1,\ldots$ \textbf{do}

\ \ \ \ \ $X_{n}:=a_{n}$

\ \ \ \ \ \textbf{for} $m=n,n-1,\ldots,0$ \textbf{do}

\ \ \ \ \ $X_{m-1}:=\left(  m-1\right)  X_{m-1}+X_{m}$

\ \ \ \ \ \textbf{end do}

\ \ \ \ \ $b_{n}:=X_{0}$

\textbf{end do}

\bigskip

\textbf{Algorithm 2. }inverse\textbf{ }Stirling transform

\textbf{Input:} $b_{m}$

\textbf{Output: }$a_{m}$

Set $Y_{n}=b_{n,m-n}$

\textbf{for} $m=0,1,\ldots$ \textbf{do}

\ \ \ \ \ $Y_{m}:=b_{m}$

\ \ \ \ \ \textbf{for} $n=m,m-1,\ldots,0$ \textbf{do}

\ \ \ \ \ $Y_{n-1}:=Y_{n}-\left(  m-n\right)  Y_{n-1}$

\ \ \ \ \ \textbf{end do}

\ \ \ \ \ $a_{m}:=Y_{0}$

\textbf{end do}

\begin{example}
Setting $a_{0,m}=1$ in (\ref{ega}), we get the well known identity \cite{rahmani}%
\[%
{\displaystyle\sum\limits_{k=0}^{n}}
\genfrac{\{}{\}}{0pt}{}{n+m}{k+m}%
_{m}=%
{\displaystyle\sum\limits_{k=0}^{m}}
s\left(  m,k\right)  B_{n+k},
\]
where $B_{n}$ is the $n$th Bell number.
\end{example}
\begin{example}
Let $\left(  F_{n}\right)  _{n\in\mathbb{N}_{0}}$ be the Fibonacci sequence
given by Binet's formula%
\[
F_{n}=\frac{1}{\sqrt{5}}\left(  \alpha^{n}-\beta^{n}\right)  ,
\]
where $\alpha=\frac{1+\sqrt{5}}{2}$ and $\beta=\frac{1-\sqrt{5}}{2}$. If the initial sequence $a_{0,m}=\frac{\left(  -1\right)  ^{m}}{\sqrt{5}%
}\left(  \left\langle  -\alpha\right\rangle _{m}-\left\langle  -\beta\right\rangle _{m}\right)  ,$
then we get the following matrix%
\[
\mathcal{S}=%
\begin{pmatrix}
0 & 1 & 0 & 1 & -4 & 19 & -108 & \cdots \\
1 & 1 & 1 & -1 & 3 & -13 & 71 & \\
1 & 2 & 1 & 0 & -1 & 6 & -37 & \\
2 & 3 & 2 & -1 & 2 & -7 & 34 & \\
3 & 5 & 3 & -1 & 1 & -1 & -3 & \\
5 & 8 & 5 & -2 & 3 & -8 & 31 & \\
8 & 13 & 8 & -3 & 4 & -9 & 28 & \\
13 & 21 & 13 & -5 & 7 & -17 & 59 & \\
\vdots &  &  &  &  &  &  &
\end{pmatrix}.
\]

From this matrix we observe that $a_{n,0}=a_{n,2}=-a_{n+2,3}=F_{n},$ and
$a_{n+3,4}=L_{n},$ where $\left(  L_{n}\right)  _{n\in\mathbb{N}_{0}}$ the
Lucas sequence given by Binet's formula%
\[
L_{n}=\alpha^{n}+\beta^{n}.
\]
It is well known that the $F_{n}$ and $L_{n}$ are connected by the formula%
\[
L_{n}=F_{n-1}+F_{n+1},\text{ }\left(  n\in\mathbb{N}\right)  .
\]

By (\ref{ega}), one can deduce that

\[%
{\displaystyle\sum\limits_{k=0}^{m}}
s\left(  m,k\right)  F_{n+k}=\frac{1}{\sqrt{5}}%
{\displaystyle\sum\limits_{k=0}^{n}}
\left(  -1\right)  ^{m+k}%
\genfrac{\{}{\}}{0pt}{}{n+m}{k+m}%
_{m}\left(  \left\langle  -\alpha\right\rangle _{m+k}-\left\langle  -\beta\right\rangle _{m+k}\right),
\]

and by Theorem \ref{th2}, we get%
\begin{align*}
F_{n}  & =%
{\displaystyle\sum\limits_{k=0}^{2}}
s\left(  2,k\right)  F_{n+k}=-F_{n+1}+F_{n+2}\\
& =-%
{\displaystyle\sum\limits_{k=0}^{3}}
s\left(  3,k\right)  F_{n+2+k}=-2F_{n+3}+3F_{n+4}-F_{n+5},
\end{align*}
and for $n\in\mathbb{N}_{0}$%
\[
L_{n}=%
{\displaystyle\sum\limits_{k=0}^{4}}
s\left(  4,k\right)  F_{n+3+k}=-6F_{n+4}+11F_{n+5}-6F_{n+6}+F_{n+7}.
\]

By Theorem \ref{th1}%
\begin{align*}
F_{n}  & =\frac{1}{\sqrt{5}}%
{\displaystyle\sum\limits_{k=0}^{n}}
\left(  -1\right)  ^{k}%
\genfrac{\{}{\}}{0pt}{}{n}{k}%
\left(  \left\langle -\alpha\right\rangle _{k}-\left\langle -\beta
\right\rangle _{k}\right)  \\
& =\frac{1}{\sqrt{5}}%
{\displaystyle\sum\limits_{k=0}^{n}}
\left(  -1\right)  ^{k}%
\genfrac{\{}{\}}{0pt}{}{n+2}{k+2}%
_{2}\left(  \left\langle -\alpha\right\rangle _{k+2}-\left\langle
-\beta\right\rangle _{k+2}\right)  \\
& =\frac{1}{\sqrt{5}}%
{\displaystyle\sum\limits_{k=0}^{n+2}}
\left(  -1\right)  ^{k}%
\genfrac{\{}{\}}{0pt}{}{n+5}{k+3}%
_{3}\left(  \left\langle -\alpha\right\rangle _{k+3}-\left\langle
-\beta\right\rangle _{k+3}\right),
\end{align*}
and
\[
L_{n}=\frac{1}{\sqrt{5}}%
{\displaystyle\sum\limits_{k=0}^{n+3}}
\left(  -1\right)  ^{k}%
\genfrac{\{}{\}}{0pt}{}{n+7}{k+4}%
_{4}\left(  \left\langle -\alpha\right\rangle _{k+4}-\left\langle
-\beta\right\rangle _{k+4}\right).
\]

\end{example}

\section{Generating function}
\begin{theorem}\label{th3}
Suppose that the initial sequence $a_{0,m+r}$ has the following exponential
generating function $A_{r}\left(  z\right)  =%
{\displaystyle\sum\limits_{k\geq0}} a_{0,k+r}\frac{z^{k}}{k!}.$ Then
the sequence $\{a_{n,r}\}_n$ of the $rth$ columns
of the matrix $\mathcal{S}$ has an exponential generating function $\mathcal{B}%
_{r}\left(  z\right)  =%
{\displaystyle\sum\limits_{n\geq0}}
a_{n,r}\dfrac{z^{n}}{n!}$ given by%
\begin{equation}
B_{r}\left(  z\right)  =e^{rz}A_{r}\left(  e^{z}-1\right)\label{anis}
\end{equation}

\end{theorem}

\begin{proof}
We have%
\begin{align*}
B_{r}\left(  z\right)   & =%
{\displaystyle\sum\limits_{k\geq0}}
a_{0,r+k}%
{\displaystyle\sum\limits_{n\geq0}}
\genfrac{\{}{\}}{0pt}{}{n+r}{k+r}%
_{r}\dfrac{z^{n}}{n!}\\
& =%
{\displaystyle\sum\limits_{k\geq0}}
a_{0,r+k}\frac{1}{k!}e^{rz}\left(  e^{z}-1\right)  ^{k}\\
& =e^{rz}%
{\displaystyle\sum\limits_{k\geq0}}
a_{0,r+k}\frac{\left(  e^{z}-1\right)  ^{k}}{k!}\\
& =e^{rz}A_{r}\left(  e^{z}-1\right)  .
\end{align*}

\end{proof}

\begin{theorem}\label{th4}
Suppose that the final sequence $a_{n+r,0}$ has the following exponential
generating function $\mathcal{B}_{r}\left(  z\right)  =%
{\displaystyle\sum\limits_{k\geq0}} a_{k+r,0}\frac{z^{k}}{k!}.$ Then
the sequence $\{a_{r,m}\}_m$ of the $rth$ rows of
the matrix $\mathcal{S}$ has an exponential generating function $\mathcal{A}%
_{r}\left(  z\right)  =%
{\displaystyle\sum\limits_{m\geq0}}
a_{r,m}\dfrac{z^{m}}{m!}$ given by%
\begin{equation}
\mathcal{A}_{r}\left(  z\right)  =\mathcal{B}_{r}\left(  \ln(1+z)\right) \label{Anis2} .
\end{equation}

\end{theorem}

\begin{proof}
We have
\begin{align*}
\mathcal{A}_{r}\left(  z\right)    & =%
{\displaystyle\sum\limits_{k\geq0}}
a_{r+k,0}%
{\displaystyle\sum\limits_{m\geq0}}
s\left(  m,k\right)  \dfrac{z^{m}}{m!}\\
& =%
{\displaystyle\sum\limits_{k\geq0}}
a_{r+k,0}\dfrac{\left(  \ln(1+z)\right)  ^{k}}{k!}\\
& =\mathcal{B}_{r}\left(  \ln(1+z)\right)  .
\end{align*}

\end{proof}

\begin{example}
A derangement on a set $\left\{  1,2,\ldots,m\right\}  $ is a permutation
$\pi=i_{1}i_{2}\cdots i_{m}$ such that $i_{k}\neq k$ for $k=1,2,\ldots m.$ The
number of derangements on $\left\{  1,2,\ldots,m\right\}  $ is denoted by
$D_{m}$ and given by $D_{m}=\left[  \frac{m!}{e}\right]  ,$ where $\left[
x\right]  $ the nearest integer function.
Now, if the initial sequence $a_{0,m}=\left(  -1\right)  ^{m}D_{m},$ then we get
the following matrix
\[%
\mathcal{S}=
\begin{pmatrix}
1 & 0 & 1 & -2 & 9 & -44 & \cdots\\
0 & 1 & 0 & 3 & -8 & 45 & \\
1 & 1 & 3 & 1 & 13 & -39 & \\
1 & 4 & 7 & 16 & 13 & 76 & \\
4 & 11 & 30 & 61 & 128 & 159 & \\
11 & 41 & 121 & 311 & 671 & 1381 & \\
\vdots &  &  &  &  &  &
\end{pmatrix}
\]
The generating function of the sequence $a_{0,m}$ is  $A_{0}(z)=\frac{e^z}{1+z}$. It follows from (\ref{anis}) that $B_{0}(z)=exp(e^{z}-z-1)$, and we notice that $a_{n,0}$ is the number $v_n$ of
partitions of $\left\{  1,2,\ldots,n\right\}  $ without singletons (see for instance \cite{sun}). By
(\ref{ega}), one can then deduce that%
\[%
{\displaystyle\sum\limits_{k=0}^{m}}
s\left(  m,k\right)  v_{n+k}=%
{\displaystyle\sum\limits_{k=0}^{n}}
\left(  -1\right)  ^{m+k}%
\genfrac{\{}{\}}{0pt}{}{n+m}{k+m}%
_{m}D_{m+k}%
\]
\bigskip If $m=0,$ we have%
\[
v_{n}=%
{\displaystyle\sum\limits_{k=0}^{n}}
\left(  -1\right)  ^{k}%
\genfrac{\{}{\}}{0pt}{}{n}{k}%
D_{k}.
\]

\end{example}

\begin{example}
The exponential generating function of the Bernoulli polynomials $B_{n}\left(
x\right)  $ is%
\[
\mathcal{B}_{0}\left(  z\right)  :=\dfrac{ze^{xz}}{e^{z}-1}=%
{\displaystyle\sum\limits_{n\geq0}}
B_{n}\left(  x\right)  \frac{z^{n}}{n!}.
\]
By Theorem 4, we have%
\[
\mathcal{A}_{0}\left(  z\right)  =\frac{\left(  1+z\right)  ^{x}\ln\left(
1+z\right)  }{z},
\]
It is not difficult to show that
\[
\left[  z^{m}\right]  \mathcal{A}_{0}\left(  z\right)  =%
{\displaystyle\sum\limits_{i=0}^{m}}
\left(  -1\right)  ^{m-i}\frac{\left(  x\right)  _{i}}{m-i+1},
\]
where $\left[  z^{n}\right]  f\left(  z\right)  $ denote the operation of
extracting the coefficient of $z^{n}$ in the formal power series $f\left(
z\right)  =%
{\displaystyle\sum}
f_{n}z^{n}$. Now, let us consider $\mathcal{S}$ defined by (\ref{rec2}) with the final sequence $a_{n,0}%
=B_{n}\left(  x\right)  ,$ by (\ref{ega}), we have%
\[%
{\displaystyle\sum\limits_{k=0}^{m}}
s\left(  m,k\right)  B_{n+k}\left(  x\right)  =%
{\displaystyle\sum\limits_{k=0}^{n}}
\genfrac{\{}{\}}{0pt}{}{n+m}{k+m}%
_{m}%
{\displaystyle\sum\limits_{i=0}^{m+k}}
\left(  -1\right)  ^{m+k-i}\frac{\left(  x\right)  _{i}}{m+k-i+1}%
\]
\end{example}
\begin{example}
Catalan and Motzkin numbers naturally appear in a large number of
combinatorial objects. It is well known that the Catalan number $C_{n}%
=\frac{1}{n+1}\binom{2n}{n}$ and Motzkin number $M_{n}=%
{\displaystyle\sum\limits_{k=0}^{\left\lfloor n/2\right\rfloor }}
\frac{1}{k+1}\binom{n}{2k}\binom{2k}{k}$ are connected by \cite{Bernhart}
\[
C_{n+1}=%
{\displaystyle\sum\limits_{k=0}^{n}}
\binom{n}{k}M_{k}\Longleftrightarrow M_{n}=%
{\displaystyle\sum\limits_{k=0}^{n}}
\left(  -1\right)  ^{n-k}\binom{n}{k}C_{k+1}.
\]
Using the generalized Stirling transform, we can show that the Catalan numbers
are related with Motzkin numbers in terms of Stirling numbers by%
\begin{equation}
{\displaystyle\sum\limits_{k=0}^{n}}
s\left(  n,k\right)  M_{k}=%
{\displaystyle\sum\limits_{k=0}^{n+1}}
s\left(  n+1,k\right)  C_{k},\label{Anis3}
\end{equation}
and%
\begin{equation}
C_{n}=\delta_{n,0}+%
{\displaystyle\sum\limits_{k=1}^{n}}
{\displaystyle\sum\limits_{i=0}^{k-1}}
\genfrac{\{}{\}}{0pt}{}{n}{k}%
s\left(  k-1,i\right)  M_{i}\Longleftrightarrow M_{n}=%
{\displaystyle\sum\limits_{k=0}^{n}}
{\displaystyle\sum\limits_{i=0}^{k+1}}
\genfrac{\{}{\}}{0pt}{}{n}{k}%
s\left(  k+1,i\right)  C_{i}.\label{Anis4}
\end{equation}

Setting the final sequence  $a_{n,0}=C_{n},$ we get the following matrix%
\[
\mathcal{S}=%
\begin{pmatrix}
1 & 1 & 1 & 1 & 0 & 1 & -5 & 29 & \cdots\\
1 & 2 & 3 & 3 & 1 & 0 & -1 & 7 & \\
2 & 5 & 9 & 10 & 4 & -1 & 1 & -1 & \\
5 & 14 & 28 & 34 & 15 & -4 & 5 & -11 & \\
14 & 42 & 90 & 117 & 56 & -15 & 19 & -42 & \\
42 & 132 & 297 & 407 & 209 & -56 & 72 & -160 & \\
132 & 429 & 1001 & 1430 & 780 & -208 & 272 & -614 & \\
\vdots &  &  &  &  &  &  &  &
\end{pmatrix}
.
\]
Since%
\[
\mathcal{B}_{0}\left(  z\right)  =%
{\displaystyle\sum\limits_{n\geq0}}
C_{n}\frac{z^{n}}{n!}=_{1}F_{1}\left(
\begin{array}
[c]{c}%
1/2\\
2
\end{array}
;4z\right)  ,
\]
where $_{1}F_{1}\left(
\begin{array}
[c]{c}%
p\\
q
\end{array}
;z\right)  =%
{\displaystyle\sum\limits_{n\geq0}}
\frac{\left\langle p\right\rangle _{n}}{\left\langle q\right\rangle _{n}%
n!}z^{n}.$ It follows from (\ref{Anis2}) that%
\[
\mathcal{A}_{0}\left(  z\right)  =%
{\displaystyle\sum\limits_{n\geq0}}
R_{n}\frac{z^{n}}{n!}=_{1}F_{1}\left(
\begin{array}
[c]{c}%
1/2\\
2
\end{array}
;4\ln\left(  1+z\right)  \right),
\]
and
\begin{equation}
{\displaystyle\sum\limits_{k=0}^{m}}
s\left(  m,k\right)  C_{n+k}=%
{\displaystyle\sum\limits_{k=0}^{n}}
\genfrac{\{}{\}}{0pt}{}{n+m}{k+m}%
_{m}R_{m+k}\label{Anis5}
\end{equation}

Now, if the initial sequence $a_{0,m}=R_{m+1},$ we get the following matrix%
\[
\mathcal{T}=%
\begin{pmatrix}
1 & 1 & 1 & 0 & 1 & -5 & 29 & -196 & \cdots\\
1 & 2 & 2 & 1 & -1 & 4 & -22 & 146 & \\
2 & 4 & 5 & 2 & 0 & -2 & 14 & -100 & \\
4 & 9 & 12 & 6 & -2 & 4 & -16 & 93 & \\
9 & 21 & 30 & 16 & -4 & 4 & -3 & -26 & \\
21 & 51 & 76 & 44 & -12 & 17 & -44 & 172 & \\
51 & 127 & 196 & 120 & -31 & 41 & -92 & 282 & \\
\vdots &  &  &  &  &  &  &  &
\end{pmatrix}.
\]
From this matrix we observe that $a_{n,0}=M_{n}.$ We prove this observation
using generating functions. We have%
\begin{align*}
A_{0}\left(  z\right)    & =%
{\displaystyle\sum\limits_{n\geq0}}
R_{n+1}\frac{z^{n}}{n!}\\
& =\frac{d}{dz}\left(
{\displaystyle\sum\limits_{n\geq0}}
R_{n}\frac{z^{n}}{n!}\right)  \\
& =\frac{1}{1+z}\text{ }_{1}F_{1}\left(
\begin{array}
[c]{c}%
3/2\\
3
\end{array}
;4\ln\left(  1+z\right)  \right)  .
\end{align*}
From (\ref{anis}), we get%
\begin{align*}
B_{0}\left(  z\right)    & =_{1}F_{1}\left(
\begin{array}
[c]{c}%
3/2\\
3
\end{array}
;4z\right)  e^{-z}\\
& =\frac{d}{dz}\left(
{\displaystyle\sum\limits_{n\geq0}}
C_{n}\frac{z^{n}}{n!}\right)
{\displaystyle\sum\limits_{n\geq0}}
\left(  -1\right)  ^{n}\frac{z^{n}}{n!}\\
& =%
{\displaystyle\sum\limits_{n\geq0}}
\left(
{\displaystyle\sum\limits_{k=0}^{n}}
\left(  -1\right)  ^{n-k}\binom{n}{k}C_{k+1}\right)  \frac{z^{n}}{n!}\\
& =%
{\displaystyle\sum\limits_{n\geq0}}
M_{n}\frac{z^{n}}{n!}.
\end{align*}
It follows
\begin{equation}
{\displaystyle\sum\limits_{k=0}^{m}}
s\left(  m,k\right)  M_{n+k}=%
{\displaystyle\sum\limits_{k=0}^{n}}
\genfrac{\{}{\}}{0pt}{}{n+m}{k+m}%
_{m}R_{m+k+1}.\label{Anis6}
\end{equation}
Combining results (\ref{Anis5}) and (\ref{Anis6}) gives (\ref{Anis3}) and (\ref{Anis4}).
\end{example}

\section{Hankel transform}
The Hankel transform of a sequence $\alpha_{n}$ is the sequence of Hankel
determinants  $\det\left(  \alpha_{i+j}\right)  _{0\leq i,j\leq n}$. A number of methods for computing the Hankel determinants have been widely
investigated \cite{yip,Kra1,Kra2,barry}. It is well known that the Hankel transform of sequences $\alpha_{n}$ and $\beta_{n}$ are equal under the binomial transform \cite{layman}

\[
\beta_{n}=%
{\displaystyle\sum\limits_{k=0}^{n}}
\binom{n}{k}\alpha_{k}.
\]

A natural question arises: "What about the Hankel transform of the sequences $a_{n}$ and $b_{n}$ under the Stirling transform?" In this section we show that there is a connection between the generalized Stirling transform and the Hankel determinants.

\begin{theorem}
For $n\geq0,$ we have%
\[
\det\left(  a_{i,j}\right)  _{0\leq i,j\leq n}=\det\left(  b_{i+j}\right)
_{0\leq i,j\leq n}.
\]

\end{theorem}
\begin{proof}
We can write
\[
\det\left(  b_{i+j}\right)  _{0\leq i\,,j\leq n}=%
\begin{vmatrix}
a_{0,0} & a_{1,0} & a_{2,0} & a_{3,0} & \cdots & a_{n,0}\\
a_{1,0} & a_{2,0} & a_{3,0} & a_{4,0} & \cdots & a_{n+1,0}\\
\vdots & \vdots & \vdots &  & \ddots & \vdots\\
a_{n,0} & a_{n+1,0} & a_{n+2,0} & a_{n+3,0} & \cdots & a_{2n,0}%
\end{vmatrix}
,
\]
\bigskip after applying (\ref{rec2}), the determinant is unchanged%
\[
\det\left(  b_{i+j}\right)  _{0\leq i\,,j\leq n}=%
\begin{vmatrix}
a_{0,0} & a_{1,0} & a_{2,0}-a_{1,0} & a_{3,0}-a_{2,0}-2\left(  a_{2,0}%
-a_{1,0}\right)   & \cdots\\
a_{1,0} & a_{2,0} & a_{3,0}-a_{2,0} & a_{4,0}-a_{3,0}-2\left(  a_{3,0}%
-a_{2,0}\right)   & \cdots\\
\vdots & \vdots & \vdots & \vdots \\
a_{n,0} & a_{n+1,0} & a_{n+2,0}-a_{n+1,0} & a_{n+3,0}-a_{n+1,0}-2\left(
a_{n+2,0}-a_{n+1,0}\right)   & \cdots
\end{vmatrix}.
\]
Using (\ref{sah}), we get%
\[
\det\left(  b_{i+j}\right)  _{0\leq i\,,j\leq n}=%
\begin{vmatrix}
a_{0,0} & a_{0,1} & a_{0,2} & a_{0,3} & \cdots & a_{0,n}\\
a_{1,0} & a_{1,1} & a_{1,2} & a_{1,3} & \cdots & a_{1,n}\\
\vdots & \vdots & \vdots &  & \ddots & \\
a_{n,0} & a_{n,1} & a_{n,2} & a_{n,3} &  & a_{n,n}%
\end{vmatrix}
,
\]
from which the relation follows.
\end{proof}
The answer to the previous question is given in the following
\begin{corollary}
For $n\in\mathbb{N}_{0},$ we have%
\[
\det\left(  b_{i+j}\right)  _{0\leq i,j\leq n}=\det\left(
{\displaystyle\sum\limits_{k=0}^{i}}
\genfrac{\{}{\}}{0pt}{}{i+j}{k+j}%
_{j}a_{k+j}\right)  _{0\leq i,j\leq n}.%
\]

\end{corollary}

\end{document}